\documentclass{article}

     \PassOptionsToPackage{numbers, compress}{natbib}

     \usepackage[preprint]{neurips_2019}

\usepackage[utf8]{inputenc} 
\usepackage[T1]{fontenc}    
\usepackage{hyperref}       
\usepackage{url}            
\usepackage{booktabs}       
\usepackage{amsfonts}       
\usepackage{nicefrac}       
\usepackage{microtype}      
\usepackage{microtype}
\usepackage{graphicx}
\usepackage[]{algorithm}
\usepackage[]{algorithmic}
\usepackage{subfigure}
\usepackage{booktabs} 
\usepackage{amsmath}
\usepackage{color}
\usepackage{amsfonts}
\RequirePackage{amssymb}
\RequirePackage{natbib}
\usepackage[compact]{titlesec}
\titlespacing{\subsection}{1pt}{*0}{*0}

\newcommand{\BlackBox}{\rule{1.5ex}{1.5ex}}  
\newenvironment{proof}{\par\noindent{\bf Proof\ }}{\hfill\BlackBox\\[2mm]}
 
\newtheorem{theorem}{Theorem}[section]
\newtheorem{lemma}{Lemma}[section]

\newtheorem{definition}{Definition}[section]

\def\R{\mathbb{R}}

\def\Eps{\mathcal{E}}

\def\1{\mathbf{1}}

\def\X{\mathcal{X}}
\def\L{\mathcal{L}}

\DeclareMathOperator*{\argmin}{arg\,min}
\DeclareMathOperator*{\argmax}{arg\,max}

\def \R{\mathbb R}

\title{Online Continuous DR-Submodular Maximization with Long-Term Budget Constraints}

\author{%
  Omid Sadeghi\\
University of Washington\\
Seattle, WA 98195\\
\texttt{omids@uw.edu} \\
\And
Maryam Fazel \\
University of Washington\\
Seattle, WA 98195 \\
\texttt{mfazel@uw.edu}\\
}

\begin{document}

\maketitle

\begin{abstract}
	In this paper, we study a class of online optimization problems with long-term budget constraints where the objective functions are not necessarily concave (nor convex) but they instead satisfy the Diminishing Returns (DR) property. Specifically, a sequence of monotone DR-submodular objective functions $\{f_t(x)\}_{t=1}^T$ and monotone linear budget functions $\{\langle p_t,x \rangle \}_{t=1}^T$ arrive over time and assuming a total targeted budget $B_T$, the goal is to choose points $x_t$ at each time $t\in\{1,\dots,T\}$, without knowing $f_t$ and $p_t$ on that step, to achieve sub-linear regret bound while the total budget violation $\sum_{t=1}^T \langle p_t,x_t \rangle -B_T$ is sub-linear as well. Prior work has shown that achieving sub-linear regret is impossible if the budget functions are chosen adversarially. Therefore, we modify the notion of regret by comparing the agent against a $(1-\frac{1}{e})$-approximation to the best fixed decision in hindsight which satisfies the budget constraint proportionally over any window of length $W$. We propose the Online Saddle Point Hybrid Gradient (OSPHG) algorithm to solve this class of online problems. For $W=T$, we recover the aforementioned impossibility result. However, when $W=o(T)$, we show that it is possible to obtain sub-linear bounds for both the $(1-\frac{1}{e})$-regret and the total budget violation. 
\end{abstract}
\section{Introduction}
\subsection{Motivating Application: Online Ad Placement}\label{app}
Consider the following online ad placement problem: At round $t\in[T]$, an advertiser should choose an investment vector $x_t \in R_+^n$ over $n$ different websites where $i$-th entry of $x_t$ denotes the amount that the advertiser is willing to pay per each click on the ad on the $i$-th website (i.e., cost per click). In other words, each website has different tiers of ads and choosing $x_t$ corresponds to ordering a certain type of ad. The aggregate cost of investment would be determined when the number of clicks the ad receives is revealed. In other words, the cost of such an investment would be $\langle p_t, x_t\rangle$ where the $i$-th entry of the vector $p_t$ is the number of clicks the ad on the $i$-th website received. Note that the vector $p_t$ is not known ahead of time and could be adversarial. For instance, competing advertisers may click on the ad to deplete their rival's budget. The advertiser needs to balance her total investment against an allotted long-term budget (daily, monthly, etc.), i.e., $\sum_{t=1}^T \langle p_t, x_t\rangle \leq B_T$ where $B_T$ is the total targeted budget. 
At round $t\in[T]$, the advertiser's utility function $f_t (x_t)$ is a monotone DR-submodular function with respect to the vector of investments and this function quantifies the overall amount of impressions of the ads. DR-submodularity of the utility function characterizes the diminishing returns property of the impressions (Diminishing Returns (DR) property and continuous DR-submodular functions are defined in section \ref{dr} at page \pageref{dr}). In other words, making an ad more visible will attract proportionally fewer extra viewers because each website shares a portion of its visitors with other websites.\\
In this paper, we aim to propose an algorithm for this class of online optimization problems such that the algorithm has no regret, i.e., sub-linear regret bound with respect to the horizon $T$, and the total budget violation is sub-linear as well.
\subsection{Related Work}
\textbf{Online convex optimization with constraints.} Consider an online problem where at step $t\in[T]$, the player chooses $x_t \in \mathcal{X}$. Then, cost function $f_t: \mathcal{X}\rightarrow \R$ and constraint function $g_t: \mathcal{X}\rightarrow \R$ are revealed and the player incurs a loss of $f_t(x_t)$ and her budget is impacted by the amount $g_t (x_t)$. $\mathcal{X}$ is assumed to be convex and compact and the functions $f_t, g_t$ are convex for all $t\in[T]$.
The overall goal is to design an algorithm whose output is asymptotically feasible, i.e. the constraint residual $\sum_{t=1}^T g_t (x_t)$ is sub-linear, and has a sub-linear regret. 
\cite{mahdavi2012trading} considered the case where all constraint functions are equal and are given offline, i.e., $g_t (x)=g(x)~\forall t\in[T],x\in\mathcal{X}$. For this setting, they achieved $\mathcal{O}(\sqrt{T})$ regret and $\mathcal{O}(T^{\frac{2}{3}})$ constraint residual (i.e., $\sum_{t=1}^T g(x)$) bounds. \cite{jenatton2015adaptive} studied the exact same framework as \cite{mahdavi2012trading} and improved upon their result by obtaining $\mathcal{O}(T^{\max\{\beta,1-\beta\}})$ regret and $\mathcal{O}(T^{1-\frac{\beta}{2}})$ constraint residual bounds where $\beta\in(0,1)$ is a free parameter. More recently, \cite{NIPS2018_7852} considered an alternative notion of constraint residual defined as the sum of squares of clipped residuals, $\sum_{t=1}^T (\max\{g(x),0\})^2$, and achieved $\mathcal{O}(T^{\max\{\beta,1-\beta\}})$ regret and $\mathcal{O}(T^{1-\beta})$ constraint residual bounds for time-invariant constraint functions. Also, they obtained logarithmic regret bound for the case that cost functions are strongly convex. The new constraint residual form considered in \cite{NIPS2018_7852} heavily penalizes large constraint violations and strictly feasible solutions of some rounds cannot cancel out the effect of violated constraints at other rounds.\\ 
For the setting with time-varying constraints, \cite{mannor2009online} considered the notion of regret with window length $W=T$ and provided a simple counterexample with linear functions showing that the regret of any causal algorithm would be lower bounded by $\Omega(T)$. \cite{neely2017online} studied general time-varying constraint functions and assuming that there exists an action $x^* \in\mathcal{X}$ such that $g_t (x^*)<0~\forall t\in[T]$ (Slater condition), they obtained $\mathcal{O}(\sqrt{T})$ bounds for both regret with window size $W=1$ and constraint residual. However, the fixed decision benchmark action considered in this paper is constrained to be feasible for all constraint functions $g_t~\forall t\in[T]$ which heavily restricts the performance of the benchmark action and thus, the obtained regret guarantees could be loose. \cite{sun2017safety} considered the same notion of regret as \cite{neely2017online} and using online mirror descent as a subroutine (and without using the Slater condition), they obtained a similar $\mathcal{O}(\sqrt{T})$ regret bound and a looser $\mathcal{O}(T^{\frac{3}{4}})$ constraint residual bound.
\cite{icml2019} considered the exact same framework and algorithm as \cite{neely2017online}, however, they constrained the fixed decision comparator to be feasible in windows of size $W=T^\beta$ where $\beta \in[0,1)$ (as opposed to \cite{neely2017online} where $W=1$). They obtained $\mathcal{O}(\frac{WT}{V}+\sqrt{T})$ regret bound and $\mathcal{O}(\sqrt{VT})$ residual bound where $V\in [W,T)$ is a tunable parameter. Note that for $W=1$, this bound does not achieve the $\mathcal{O}(\sqrt{T})$ regret and constraint residual bound of \cite{neely2017online}. \cite{chen2017online,cao2018online} considered an alternative dynamic notion of regret defined as $R_T = \sum_{t=1}^T f_t (x_t)-\min_{\{\tilde{x}_t\in \mathcal{X}:~g_t(\tilde{x}_t)\leq 0~\forall t\in[T]\}} \sum_{t=1}^T f_t (\tilde{x}_t)$. Provided that the drift of the benchmark sequence, i.e., $\sum_{t=2}^T \|x^*_{t-1}-x^*_t\|_2$, is sub-linear in $T$, they achieved sub-linear regret and constraint residual bounds. However, the required assumption on the benchmark sequence having a sub-linear drift is really strong and is difficult to guarantee.\\
Note that in all these works, the objective functions are assumed to be convex. In contrast, we consider a more general class of non-convex/non-concave continuous DR-submodular objective functions to which the aforementioned results are not applicable.\\
\textbf{Online submodular maximization.} An orthogonal research direction considers the following problem: At step $t\in \{1,\dots,T\}$, the online algorithm chooses a feasible point $x_t \in \mathcal{P}$. Once the algorithm commits to this choice, a monotone continuous DR-submodular function $f_t$ is revealed and the reward $f_t (x_t)$ is received. The goal is to minimize the regret defined as the difference between the total reward obtained by the algorithm and that of the $(1-\frac{1}{e})$ approximation to the best fixed decision in hindsight with $(1-\frac{1}{e})$ being the optimal approximation ratio for an offline monotone continuous DR-submodular maximization problem. Note that although similar to our framework (the objective functions are assumed to be continuous DR-submodular in this setting), there are no time-varying constraints arriving online and therefore, they do not deal with the considerable complication of bounding the constraint residual.\\
The meta-algorithm for online submodular maximization problem is as follows:
\begin{algorithm}[H]
	\caption{Online submodular maximization meta-algorithm \cite{online-Chen}}
	\begin{algorithmic}
		\STATE Input: $\mathcal{P}$ is a convex set and $T$ is the horizon
		\STATE Output: $\{x_t :1\leq t\leq T\}$
		\STATE Choose an off-the-shelf online linear maximization algorithm and initialize $K$ instances $\mathcal{E}_k~\forall k\in[1,K]$ of it for online maximization of linear utility functions over $\mathcal{P}$
		\FOR{$t=1$ {\bfseries to} $T$}
		\STATE Set $x_t(1)=0$.
		\FOR{$k=1$ {\bfseries to} $K$}
		\STATE Let $v_t (k)$ be the vector selected by $\mathcal{E}_k$
		\STATE $x_t(k+1)=x_t(k)+\frac{1}{K}v_t(k)$
		\ENDFOR
		\STATE Play $x_t=x_t(K+1)$, observe the function $f_t$ and the reward $f_t (x_t)$
		\STATE Feedback $\langle v_t (k),\nabla f_t (x_t (k))\rangle$ as the payoff to be received by $\mathcal{E}_k$
		\ENDFOR
	\end{algorithmic}
\end{algorithm}
For each instance $\mathcal{E}_k~\forall k\in[1,K]$, at round $t\in[T]$, the algorithm chooses the point $v_t (k)$ and then receives the linear reward $\langle v_t (k), \nabla f_t \big(\frac{1}{K}\sum_{s=1}^{k-1}v_t (s)\big) \rangle$. \cite{online-Golovin} considered the case that the continuous DR-submodular function $f_t$ is the multilinear extension of a discrete submodular function and $\mathcal{P}$ is the matroid polytope. Using the Perturbed Follow the Leader (PFTL) as the online algorithm, they achieved an $O(\sqrt{T})$ $(1-\frac{1}{e})$-regret bound. \cite{online-Chen} used Regularized Follow The Leader (RFTL) online algorithm and achieved an $O(\sqrt{T})$ $(1-\frac{1}{e})$-regret bound for general continuous DR-submodular functions. In \cite{stochasticonline-Chen}, they further generalized their result and developed a projection-free algorithm which only requires stochastic gradient estimates and achieves a similar regret bound. See \cite{survey-Krause} for a detailed overview of online maximization of submodular functions.\\
\subsection{Contributions}
In this paper, we aim to design an algorithm for online continuous DR-submodular maximization problem with long-term budget constraints to achieve sub-linear regret and budget violation bounds simultaneously. Specifically, we make the following contributions:
\begin{itemize}
	\item We introduce the online continuous DR-submodular maximization problem with long-term budget constraints. The online ad placement example mentioned in section \ref{app} is an application of this framework.
	\item We propose the Online Saddle Point Hybrid Gradient (OSPHG) algorithm to solve this class of online problems. Our algorithm is inspired by that of \cite{sun2017safety} and \cite{online-Chen}. We consider a refined notion of static regret where the agent's utility is compared against a $(1-\frac{1}{e})$ approximation to the best fixed decision in hindsight which satisfies the budget constraint proportionally over any window of length $W$. For $W=T$, we recover the known impossibility result obtained by \cite{mannor2009online}. However, for $W=o(T)$, we obtain sub-linear bounds for both the $(1-\frac{1}{e})$-regret and the total budget violation. In particular, if $W=T^{1-\epsilon}~0< \epsilon\leq 1$, we obtain a $(1-\frac{1}{e})$-regret bound $\mathcal{O}(T^{1-\frac{\epsilon}{2}})$ while the total budget violation is $\mathcal{O}(T^{1-\frac{\epsilon}{4}})$.
\end{itemize}
Finally, we validate our theoretical results through conducting numerical experiments for a class of non-convex/non-concave continuous DR-submodular objective functions.
\section{Preliminaries}
\subsection{Notation}
We will use $[T]$ to denote the set $\{1,2,\dots,T\}$. For a vector $u\in\mathbb{R}^n$, we define $[u]_+ :=\max\{u,0\}$ and $[u]_- :=\min\{u,0\}$. The inner product of two vectors $x,y\in\mathbb{R}^n$ is denoted by either $\langle x, y \rangle$ or $x^T y$. Also, for two vectors $x,y\in \mathbb{R}^n$, $x\preceq y$ implies that $x_i \leq y_i~\forall i\in[n]$. For a vector $x\in \R^n$, we use $\|x\|$ to denote the Euclidean norm of $x$. For a convex set $\X$, we will use $\mathcal{P}_{\X}(y)=\argmin_{x\in \X}\|x-y\|$ to denote the projection onto set $\X$.
\subsection{Diminishing Returns (DR) property}\label{dr}
\begin{definition}\label{def:dr}
	A differentiable function $F:K \rightarrow \mathbb{R}$, $K\subset \mathbb{R}_+^n$, satisfies the Diminishing Returns (DR) property if:
	\begin{equation}
	x\succeq y \Rightarrow \nabla F(x) \preceq \nabla F(y)\nonumber
	\end{equation}
	In other words, $\nabla F$ is an anti-tone mapping from $\mathbb{R}^n$ to $\mathbb{R}^n$.\\
	If $F$ is twice differentiable, DR property is equivalent to the Hessian matrix being element-wise non-positive. Note that for $n=1$, the DR property is equivalent to concavity. However, for $n>1$, concavity implies negative semi-definiteness of the Hessian matrix which is not equivalent to the Hessian matrix being element-wise non-positive.
\end{definition}
A similar property is introduced in \cite{swm-Vondrak} and \cite{dr-Bian} as well and functions satisfying this property are called ``smooth submodular'' and ``DR-submodular'' there respectively. Additionally, \cite{online-Eghbali} defined the DR property for concave functions with respect to a partial ordering induced by a cone and showed that by taking the cone to be $\mathbb{R}_+^n$, Definition \ref{def:dr} is recovered and if the cone of positive semi-definite matrices is considered, the DR property generalizes to matrix ordering as well \cite{onlinematrix-Eghbali}. \cite{dr-Bian} showed that DR-submodular functions are concave along any non-negative direction, and any non-positive direction. In other words, for a DR-submodular function $F$, if $t\geq 0$ and $v\in \mathbb{R}^n$ satisfies $v\succeq 0$ or $v\preceq 0$, we have:
\begin{equation*}
F(x+tv)\leq F(x)+t\langle \nabla F(x),v\rangle 
\end{equation*}
\subsection{Examples of continuous non-concave DR-submodular functions}\label{examples}
\textbf{Multilinear extension of discrete submodular functions. \cite{multilinearmax-calinescu}} A discrete function $f:\{0,1\}^V\rightarrow \mathbb{R}$ is submodular if for all $j\in V$ and $A\subseteq B\subseteq V\setminus\{j\}$, the following holds:
\begin{equation*}
f(A\cup\{j\})-f(A)\geq f(B\cup\{j\})-f(B)
\end{equation*} 
The multilinear extension $F:[0,1]^V \rightarrow \mathbb{R}$ of $f$ is defined as:
\begin{equation*}
F(x)=\sum_{S\subset V}f(S)\prod_{i\in S}x_i \prod_{j\notin S}(1-x_j)=\mathbb{E}_{S\sim x}[f(S)]
\end{equation*}
Multilinear extensions are extensively used for maximizing their corresponding submodular set function and are known to be a special case of non-concave DR-submodular functions. The Hessian matrix of this class of functions has non-positive off-diagonal entries and all its diagonal entries are zero. It has been shown that for a large class of submodular set functions, their multilinear extension could be efficiently computed (see \cite{sub-Iyer,drapp-Bian} for examples and details).\\
\textbf{Non-convex/non-concave quadratic functions.} Consider the quadratic function $F(x)=\frac{1}{2}x^T Hx+h^T x+c$. If the matrix $H$ is element-wise non-positive, $F$ would be a DR-submodular function. We use this class of non-concave DR-submodular functions for the numerical experiments.\\
See \cite{dr-Bian,dralg-Bian} for more examples of continuous DR-submodular objective functions.
\section{Problem Statement} 
The overall offline optimization problem is the following:
\begin{equation}\label{prob}
\begin{array}{ll}
\mbox{maximize}_{x_t \in \X}& \sum_{t=1}^T f_t (x_t)\\
\mbox{subject to}& \sum_{t=1}^T g_t (x_t)\leq 0\\
\end{array}
\end{equation}
The online framework is as follows: At step $t\in[T]$, the player chooses $x_t \in \mathcal{X}$. Then, utility function $f_t: \mathcal{X}\rightarrow \R$ and constraint function $g_t: \mathcal{X}\rightarrow \R$, where $g_t(x)=\langle p_t,x\rangle -\frac{B_T}{T}$, are revealed and the player obtains the reward $f_t(x_t)$ and her budget is impacted by the amount $\langle p_t, x_t \rangle$. It is assumed that $\X\subset \R_+^n$ is convex and compact. For all $t\in[T]$, $f_t: \X\to \mathbb{R}$ is a differentiable normalized monotone continuous  DR-submodular function and $g_t: \X \to \mathbb{R}$ is linear and monotone, i.e., $p_t\succeq 0$.\\
\subsection{Performance Metric}
In order to quantify the performance of our proposed algorithm, we first define our notion of regret and total budget violation below:
\begin{definition}[Regret Metric]
	The $(1-\frac{1}{e})$-regret is defined as:
	\begin{equation*}
	R_T = (1-\frac{1}{e})\sum_{t=1}^T f_t (x_W^*)-\sum_{t=1}^T f_t (x_t)
	\end{equation*}
	where:
	\begin{align*}
	x_W^*&=\argmax_{x\in \X_W} \sum_{t=1}^T f_t (x)\\
	\X_W &= \{x\in \X:\sum_{\tau=t}^{t+W-1}g_{\tau}(x)\leq 0,~1\leq t\leq T-W+1\}
	\end{align*}
\end{definition}
$R_T$ measures the difference between the output of the algorithm and a $(1-\frac{1}{e})$ approximation to the best fixed decision in hindsight which is feasible over all windows of length $W$. Note that very recently, \cite{icml2019} first introduced the notion of a ``$K$-benchmark'', i.e., a comparator which meets the problem's allotted budget over any window of length $K$, and used this notion for online convex problems with time-varying constraints. 
\begin{definition}[Total Budget Violation Metric]
	The total budget violation is defined as follows:
	\begin{equation*}
	C_T=\sum_{t=1}^T g_t (x_t)=\sum_{t=1}^T \langle p_t,x_t \rangle -B_T
	\end{equation*}
\end{definition}
We aim to design online algorithms which achieve sub-linear bounds for both the $(1-\frac{1}{e})$-regret $R_T$ and the budget violation $C_T$.
\subsection{Assumptions}
We make the following assumptions:
\begin{itemize}
	\item $\X \subset \R_+^n$ is a compact and convex set and it contains the origin, i.e., $0\in \X$.
	\item The bounded diameter of the compact set $\X$ is $R$, i.e., we have:
	\begin{equation*}
		{\rm diam}(\X):=\max_{x,y \in \X}\|y-x\|\leq R
	\end{equation*}
	\item Both the utility functions $f_t~\forall t\in[T]$ and constraint functions $g_t~\forall t\in[T]$ are Lipschitz continuous with parameters $\beta_f$ and $\beta_g$ respectively and $\beta=\max\{\beta_f,\beta_g\}$. In other words, for all $x,y\in \X$ and $t\in[T]$, we have:
	\begin{align*}
		|f_t (y)-f_t(x)|&\leq \beta_f \|y-x\|\\
		|g_t (y)-g_t(x)|&\leq \beta_g \|y-x\|
	\end{align*}
	Note that since $g_t$ is linear for all $t\in[T]$, $\beta_g =\max_{t\in[T]} \|p_t\|$ holds.
	\item Using previous assumptions, we have:
	\begin{align*}
	F&:=\max_{t\in[T]}\max_{x,y\in\mathcal{X}}|f_t (x)-f_t (y)|\leq \beta_f R\\
	G&:=\max_{t\in[T]}\max_{x\in\mathcal{X}}|g_t (x)|\leq \beta_g R-\frac{B_T}{T}\\
	\end{align*}
	\item For all $t\in[T]$, the utility functions $f_t$ are $L$-smooth, i.e., for all $t\in[T]$, $x\in \X$ and $u\in \mathbb{R}^n$ where $u\succeq 0$ or $u\preceq 0$, the following holds:
	\begin{equation*}
	f_t(x+u)-f_t(x)\geq \langle u, \nabla f_t (x)\rangle -\frac{L}{2}\|u\|^2\nonumber
	\end{equation*} 
\end{itemize}
\section{Online Saddle Point Hybrid Gradient (OSPHG): Algorithm and Analysis}
\subsection{Algorithm}
Consider the Online Saddle Point Hybrid Gradient (OSPHG) algorithm below:
\begin{algorithm}[H]
	\caption{Online Saddle Point Hybrid Gradient (OSPHG) algorithm}
	\begin{algorithmic}
		\STATE \textbf{Input:} $\X$ is the domain set and $T$ is the horizon, $\mu, \delta$ and $K$
		\STATE \textbf{Output:} $\{x_t :1\leq t\leq T\}$
		\STATE Initialize $K$ instances $\Eps_k~\forall k\in[K]$ of Online Gradient Ascent with step size $\mu$ for online maximization of linear functions over $\mathcal{X}$
		\STATE $\lambda_1=0$
		\FOR{$t=1$ {\bfseries to} $T$}
		\STATE $x_t^{(1)}=0$
		\FOR{$k=1$ {\bfseries to} $K$}
		\STATE Let $v_t^{(k)}$ be the output of oracle $\mathcal{E}_k$ in round $t-1$
		\STATE $x_t^{(k+1)}=x_t^{(k)}+\frac{1}{K} v_t^{(k)}$
		\ENDFOR
		\STATE Play $x_t =x_t^{(K+1)}$ and observe the function $\L_t(x_t,\lambda_t)=f_t(x_t)-\lambda_t g_t(x_t)+\frac{\delta \mu}{2}\lambda_t^2$
		\FOR{$k=1$ {\bfseries to} $K$}
		\STATE Feedback $\langle v_t^{(k)},\nabla_x \L_t (x_t^{(k)},\lambda_t)\rangle$ as the payoff to be received by $\mathcal{E}_k$
		\ENDFOR
		\STATE $\lambda_{t+1}=[\lambda_t -\mu \nabla_{\lambda}\L_t (x_t,\lambda_t)]_+$
		\ENDFOR
	\end{algorithmic}
\end{algorithm}
The OSPHG algorithm could be interpreted as running two no-regret procedures:
\begin{enumerate}
	\item $K$ instances $\Eps_k$ of Online Gradient Ascent where for each $k\in[K]$, at online step $t\in[T]$, the algorithm chooses the point $v_t^{(k)}$ and after committing to this choice, it receives a reward of $\langle v_t^{(k)},\nabla_x \L_t (x_t^{(k)},\lambda_t)\rangle$. Note that each instance $\Eps_k~\forall k\in[K]$ corresponds to an online linear maximization problem. The update for $v_{t+1}^{(k)}$ is as follows:
	\begin{equation*}
		v_{t+1}^{(k)}=\mathcal{P}_{\X}\big(v_{t}^{(k)}+\mu \nabla_x \L_t (x_t^{(k)},\lambda_t)\big)
	\end{equation*}
	where $\mathcal{P}_{\X}$ is the projection onto set $\X$. Note that in our applications, the domain set $\X$ is usually a box constraint or the simplex and therefore, projection on $\X$ can be efficiently computed.
	\item Online Gradient Descent for the sequence of losses $\{\L_t(x_t,\lambda)\}_{t=1}^T$ where at each online step $t\in[T]$, the algorithm chooses $\lambda_t\geq 0$ and then, observes the loss $-\lambda_t g_t(x_t)+\frac{\delta \mu}{2}\lambda_t^2$. Note that this is an online quadratic minimization problem. 
\end{enumerate}
Therefore, the OSPHG algorithm is in fact solving an online saddle point problem and hence the name. It is noteworthy that although we used Online Gradient Descent/Ascent as subroutines in the OSPHG algorithm, \textit{any} other off-the-shelf no-regret online optimization algorithm (such as Online Mirror Descent, Regularized Follow the Leader, etc.) could have been used instead and similar bounds would have been derived. Potential advantages of any such no-regret algorithm over the other could indeed be an interesting research direction.\\ 
If for all $t\in[T]$, $g_t(x)$ were available offline, instead of running the Online Gradient Descent for updating $\lambda$ at each step, we could have minimized $\L_t(x_t,\lambda)$ with respect to $\lambda$ to obtain $\min_{\lambda} \L_t(x_t,\lambda)=f_t(x_t)-\frac{\delta \mu}{2}g_t^2(x_t)$ which is similar to the quadratic penalty function \cite{nocedal2006numerical}.\\
\subsection{Analysis}
In order to prove the regret and budget violation bounds, we first provide Lemma \ref{l1}, \ref{l2} and \ref{ll3}.
\begin{lemma}\label{l1}
	For all $t\in[T]$, the following holds:
	\begin{equation*}
	\mu\sum_{s=1}^t (1-\delta \mu^2)^{t-s} g_s (x_s)\leq \lambda_{t+1} \leq \mu\sum_{s=1}^t (1-\delta \mu^2)^{t-s} |g_s (x_s)|
	\end{equation*}
\end{lemma}
\begin{proof}
	See Appendix A for the proof.
\end{proof}
Using Lemma \ref{l1} and the inequality $1-\delta \mu^2 \leq 1$, we can conclude that for all $t\in[T]$, $\lambda_{t+1}\leq \mu t G$ holds. We will use this fact multiple times in the proofs.
\begin{lemma}\label{l2}
	For a fixed $t\in\{1,\dots,T-W+1\}$, the following holds:
	\begin{equation*}
	\sum_{\tau=0}^{W-1}\lambda_{t+\tau}g_{t+\tau}(x_W^*)\leq \lambda_t \sum_{\tau=0}^{W-1}g_{t+\tau}(x_W^*)+\frac{G^2}{2}\mu W(W-1)
	\end{equation*}
\end{lemma}
\begin{proof}
	See Appendix B for the proof.
\end{proof}
\begin{lemma}\label{ll3}
	For $\mu=\frac{R}{\beta\sqrt{WT}}$, $\delta=4\beta^2$ and any $\lambda \geq 0$, if $T$ is large enough, we have:
	\begin{align}\label{l3}
	R_T+C_T \lambda-\frac{\delta \mu}{2}T\lambda^2-\frac{\lambda^2}{\mu}&\leq (F+\beta R)(W-1)+\frac{G}{2}(G+\beta R)\mu (W-1)(T-1)\\
	&+\frac{R^2}{\mu}+(G^2+\beta^2) \mu T+\frac{G^2}{2}\mu (W-1)(T-W+1)\nonumber\\
	&+\frac{LR^2}{2K}(T-W+1)\nonumber
	\end{align}
\end{lemma}
\begin{proof}
	See Appendix C for the proof.
\end{proof}
Now, we have all the required tools to prove the performance bounds of the OSPHG algorithm.
\begin{theorem}[Regret bound]\label{thm1}
	For $W=o(T)$, if we choose $\mu=\frac{R}{\beta\sqrt{WT}}=\mathcal{O}(\frac{1}{\sqrt{WT}})$ and $K=\mathcal{O}(\sqrt{\frac{T}{W}})$, the $(1-\frac{1}{e})$-regret $R_T$ satisfies the following:
	\begin{equation*}
		R_T \leq \mathcal{O}(\sqrt{WT})
	\end{equation*}
	Thus, for $W=T^{1-\epsilon}~\forall \epsilon >0$, the $(1-\frac{1}{e})$-regret of the OSPHG algorithm is $\mathcal{O}(T^{1-\frac{\epsilon}{2}})$ and hence sub-linear.  
\end{theorem}
\begin{proof}
	If we plug in $\lambda=0$, $\mu=\frac{R}{\beta\sqrt{WT}}=\mathcal{O}(\frac{1}{\sqrt{WT}})$ and $K=\mathcal{O}(\sqrt{\frac{T}{W}})$ in inequality \ref{l3}, the dominating terms on the right hand side of the inequality are $\frac{G}{2}(G+\beta R)\mu (W-1)(T-1)=\mathcal{O}(\sqrt{WT})$, $\frac{R^2}{\mu}=\mathcal{O}(\sqrt{WT})$, $\frac{G^2}{2}\mu (W-1)(T-W+1)=\mathcal{O}(\sqrt{WT})$ and $\frac{LR^2}{2K}(T-W+1)=\mathcal{O}(\sqrt{WT})$ and therefore, the result follows.
\end{proof}
\begin{theorem}[Budget violation bound]\label{thm2}
	For $W=o(T)$, if we choose $\mu=\frac{R}{\beta\sqrt{WT}}=\mathcal{O}(\frac{1}{\sqrt{WT}})$ and $K=\mathcal{O}(\sqrt{\frac{T}{W}})$, $C_T$ is bounded as follows:
	\begin{equation*}
	C_T \leq \mathcal{O}(W^{\frac{1}{4}}T^{\frac{3}{4}})
	\end{equation*}
	Therefore, for $W=T^{1-\epsilon}~\forall \epsilon >0$, the OSPHG algorithm achieves a sub-linear budget violation bound of $\mathcal{O}(T^{1-\frac{\epsilon}{4}})$.
\end{theorem}
\begin{proof}
	First, we observe that by assumption, $R_T \geq -FT$ holds.
	Assume that $C_T\geq 0$ (otherwise, we are done). Setting $\lambda=\frac{C_T}{\delta \mu T+\frac{2}{\mu}}$ in inequality \ref{l3}, we obtain:
	\begin{align*}
	\frac{C_T^2}{2\delta \mu T+\frac{4}{\mu}}&\leq FT+(F+\beta R)(W-1)+\frac{G}{2}(G+\beta R)\mu (W-1)(T-1)\\
	&+\frac{R^2}{\mu}+(G^2+\beta^2) \mu T+\frac{G^2}{2}\mu (W-1)(T-W+1)\nonumber\\
	&+\frac{LR^2}{2K}(T-W+1)\nonumber
	\end{align*}
	Plugging in $\mu=\frac{R}{\beta\sqrt{WT}}=\mathcal{O}(\frac{1}{\sqrt{WT}})$ and $K=\mathcal{O}(\sqrt{\frac{T}{W}})$ in the above inequality and multiplying both sides by $2\delta \mu T+\frac{4}{\mu}$, the dominating term on the right hand side of the inequality is $FT(2\delta \mu T+\frac{4}{\mu})=\mathcal{O}(W^{\frac{1}{2}}T^{\frac{3}{2}})$. Therefore, $C_T^2 \leq \mathcal{O}(W^{\frac{1}{2}}T^{\frac{3}{2}})$ holds. Taking the square root of both sides, we obtain the desired result. 
\end{proof}
Theorem \ref{thm1} and \ref{thm2} provide the first sub-linear regret and total budget violation bounds for the online submodular maximization problem with long-term budget constraints.
\begin{figure}[t]
	\centering
	\subfigure[]{\includegraphics[width=0.33\textwidth]{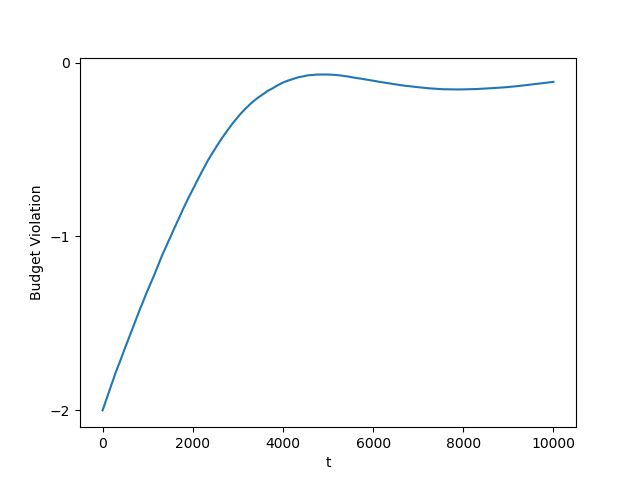}}
	\subfigure[]{\includegraphics[width=0.33\textwidth]{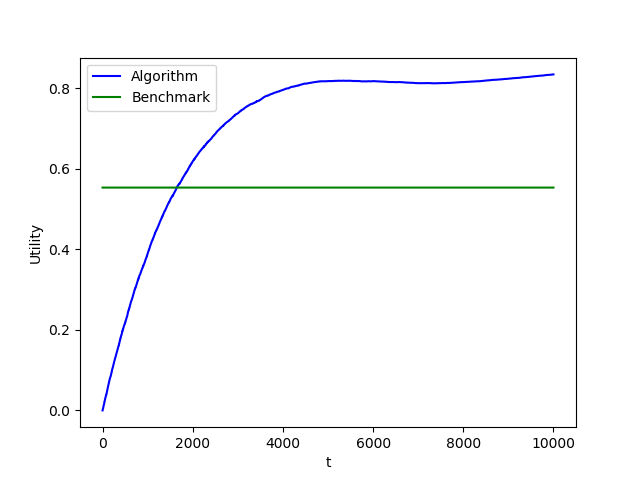}}
	\subfigure[]{\includegraphics[width=0.325\textwidth]{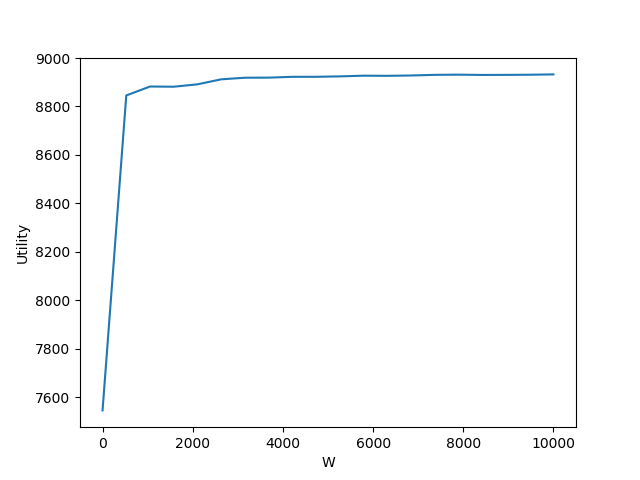}}
	\caption{(a) Budget violation running average $\frac{\sum_{\tau=1}^t g_{\tau}(x_{\tau})}{t}$ of OSPHG algorithm for $W=\sqrt{T}$ (b) Utility performance running average $\frac{\sum_{\tau=1}^t f_{\tau}(x_{\tau})}{t}$ of OSPHG algorithm for $W=\sqrt{T}$ vs. utility of the benchmark (c) Utility of the benchmark for different window lengths $1\leq W\leq T$}
	\label{fig:1}
\end{figure}
\section{Experiments}
We defined $\X=\{x\in \mathbb{R}^n:0\preceq x\preceq \bf{1}\}$ and for all $t\in[T]$, we randomly generated monotone non-convex/non-concave quadratic utility functions of the form $f_t(x)=\frac{1}{2}x^T H_tx+h_t^T x$ (see section \ref{examples}) where $H_t\in \mathbb{R}^{n\times n}$ is a random matrix with uniformly distributed non-positive entries in $[-1,0]$ and $h_t=-H_t^T \bf{1}$ to make the gradient non-negative. Therefore, the utility functions are of the form $f_t(x)=(\frac{1}{2}x-{\bf{1}})^T H_tx$. For all $t\in[T]$, we generated random linear budget functions such that $p_t$ has uniformly distributed entries in $[2,4]$. We set $T=10000$, $n=2$, $B_T=2T$ and $K=100$. We ran the OSPHG algorithm for $W=\sqrt{T}$. All codes were implemented in Python $3.7$ and the program was executed on a standard laptop computer. The running average of the budget violation and utility of the OSPHG algorithm is depicted in Figure \ref{fig:1} which verifies sub-linearity of the total budget violation and regret of our algorithm (note that the average total budget violation is negative and also, the algorithm achieves higher utilities compared to the benchmark). Additionally, we used the Frank-Wolfe variant algorithm of \cite{dr-Bian} with $K=100$ for solving offline constrained DR-submodular optimization problems to obtain the utility performance of the benchmark for different window lengths. As it could be seen in Figure \ref{fig:1}, choosing larger window sizes leads to higher utility performance for the corresponding benchmark and hence, tighter regret guarantees are obtained. However, for large enough $W$, there is merely a small difference between the obtained benchmark utility versus the case that $W=T$.
\section{Conclusion} In this paper, we studied a class of online optimization problems with long-term linear budget constraints where the utility functions are monotone continuous DR-submodular. We proposed the Online Saddle Point Hybrid Gradient (OSPHG) algorithm to solve such problems. We considered a refined notion of static regret and proved sub-linear $(1-\frac{1}{e})$-regret and budget violation bounds. Finally, we verified our theoretical findings through numerical experiments on a class of continuous DR-submodular functions. 
\appendix
\addcontentsline{toc}{section}{Appendices}
\section*{Appendices}
\section{Proof of Lemma 4.1}
Since $\lambda_{t+1}=[\lambda_t -\mu \nabla_{\lambda}\L_t(x_t,\lambda_t)]_+=[(1-\delta \mu^2)\lambda_t +\mu g_t (x_t)]_+$ and $\lambda_1=0$, we have:
\begin{align*}
	\lambda_{t+1}&\geq (1-\delta \mu^2)\lambda_t +\mu g_t (x_t)\\
	&\geq (1-\delta \mu^2)^2\lambda_{t-1} +\mu g_t (x_t)+(1-\delta \mu^2)\mu g_{t-1}(x_{t-1})\\
	&\geq \mu\sum_{s=1}^t (1-\delta \mu^2)^{t-s} g_s (x_s)+(1-\delta \mu^2)^t\underbrace{\lambda_1}_{=0}\\
	&= \mu\sum_{s=1}^t (1-\delta \mu^2)^{t-s} g_s (x_s)
\end{align*}
Similarly, we can derive the other inequality as follows:
\begin{align*}
	\lambda_{t+1}&\leq |(1-\delta \mu^2)\lambda_t +\mu g_t (x_t)|\\
	&\leq (1-\delta \mu^2)\lambda_t +\mu |g_t (x_t)|\\
	&\leq (1-\delta \mu^2)^2\lambda_{t-1} +\mu |g_t (x_t)|+(1-\delta \mu^2)\mu |g_{t-1}(x_{t-1})|\\
	&\leq \mu\sum_{s=1}^t (1-\delta \mu^2)^{t-s} |g_s (x_s)|+(1-\delta \mu^2)^t\underbrace{\lambda_1}_{=0}\\
	&= \mu\sum_{s=1}^t (1-\delta \mu^2)^{t-s} |g_s (x_s)|
\end{align*}
\section{Proof of Lemma 4.2}
Using the definition of $\lambda_{t+\tau}$, we have:
\begin{align*}
	\sum_{\tau=0}^{W-1}\lambda_{t+\tau}g_{t+\tau}(x_W^*)&=\sum_{\tau=0}^{W-1}\lambda_{t+\tau}[g_{t+\tau}(x_W^*)]_{+}+\sum_{\tau=0}^{W-1}\lambda_{t+\tau}[g_{t+\tau}(x_W^*)]_{-}\\
	&\overset{\text{(a)}}\leq \sum_{\tau=0}^{W-1}\big((1-\delta \mu^2)^{\tau}\lambda_t +\mu\sum_{i=0}^{\tau-1}(1-\delta \mu^2)^{\tau -i-1}|g_{t+i}(x_{t+i})| \big)[g_{t+\tau}(x_W^*)]_{+}\\
	&+\sum_{\tau=0}^{W-1}\big((1-\delta \mu^2)^{\tau}\lambda_t +\mu\sum_{i=0}^{\tau-1}(1-\delta \mu^2)^{\tau -i-1}g_{t+i}(x_{t+i})\big)[g_{t+\tau}(x_W^*)]_{-}\\
	&= \lambda_t \big(\sum_{\tau=0}^{W-1}(1-\delta \mu^2)^{\tau}[g_{t+\tau}(x_W^*)]_{+} +\sum_{\tau=0}^{W-1}(1-\delta \mu^2)^{\tau}[g_{t+\tau}(x_W^*)]_{-}\big)\\
	&+\mu\sum_{\tau=0}^{W-1}\big(\sum_{i=0}^{\tau-1}(1-\delta \mu^2)^{\tau-i-1}|g_{t+i}(x_{t+i})| \big)[g_{t+\tau}(x_W^*)]_{+}\\
	&+\mu\sum_{\tau=0}^{W-1}\big(\sum_{i=0}^{\tau-1}(1-\delta \mu^2)^{\tau-i-1} g_{t+i}(x_{t+i}) \big)[g_{t+\tau}(x_W^*)]_{-}\\
	&\overset{\text{(b)}}\leq \lambda_t \sum_{\tau=0}^{W-1}\underbrace{(1-\delta \mu^2)^{\tau}}_{\leq 1}g_{t+\tau}(x_W^*)+\mu\sum_{\tau=0}^{W-1}\big(\underbrace{\sum_{i=0}^{\tau-1}\underbrace{(1-\delta \mu^2)^{\tau-i-1}}_{\leq 1} |g_{t+i}(x_{t+i})|}_{\leq \tau G} \big)\underbrace{|g_{t+\tau}(x_W^*)|}_{\leq G}\\
	&\leq \lambda_t \sum_{\tau=0}^{W-1}g_{t+\tau}(x_W^*)+\frac{G^2}{2}\mu W(W-1)
\end{align*}
where (a) is due to Lemma 4.1 and (b) follows from $[g_{t+\tau}(x_W^*)]_{+}+[g_{t+\tau}(x_W^*)]_{-}=g_{t+\tau}(x_W^*)$. We will choose parameters $\delta$ and $\mu$ such that $\delta \mu^2 \ll 1$ holds.
\section{Proof of Lemma 4.3}
Fix $k\in[K]$. Using $L$-smoothness of the function $\L_t$, we have:
\begin{align*}
	\L_t (x_t^{(k+1)},\lambda_t)&\geq \L_t(x_t^{(k)},\lambda_t)+\frac{1}{K}\langle \nabla_x \L_t (x_t^{(k)},\lambda_t),v_t^{(k)}\rangle-\frac{L}{2K^2}\|v_t^{(k)}\|_2^2\\
	&\overset{\text{(a)}}\geq \L_t(x_t^{(k)},\lambda_t)+\frac{1}{K}\langle \nabla_x \L_t (x_t^{(k)},\lambda_t),v_t^{(k)}\rangle-\frac{LR^2}{2K^2}\\
	&= \L_t(x_t^{(k)},\lambda_t)+\frac{1}{K}\langle \nabla_x \L_t (x_t^{(k)},\lambda_t),v_t^{(k)}-x_W^*\rangle+\frac{1}{K}\langle \nabla_x \L_t (x_t^{(k)},\lambda_t),x_W^*\rangle-\frac{LR^2}{2K^2}\\
	&= \L_t(x_t^{(k)},\lambda_t)+\frac{1}{K}\langle \nabla_x \L_t (x_t^{(k)},\lambda_t),v_t^{(k)}-x_W^*\rangle+\frac{1}{K}\langle \nabla f_t (x_t^{(k)}),x_W^*\rangle\\
	&-\frac{1}{K}\lambda_t \langle \nabla g_t(x_t^{(k)}),x_W^*\rangle-\frac{LR^2}{2K^2}\\
	&\overset{\text{(b)}}= \L_t(x_t^{(k)},\lambda_t)+\frac{1}{K}\langle \nabla_x \L_t (x_t^{(k)},\lambda_t),v_t^{(k)}-x_W^*\rangle+\frac{1}{K}\langle \nabla f_t (x_t^{(k)}),x_W^*\rangle\\
	&-\frac{1}{K}\lambda_t g_t (x_W^*)-\frac{1}{K} \lambda_t \frac{B_T}{T}-\frac{LR^2}{2K^2}
\end{align*} 
where (a) is due to the assumption that ${\rm diam}(\X)\leq R$. Note that in order to obtain (b), we have used linearity of the budget functions for all $t\in[T]$ to write $\langle \nabla g_t(x_t^{(k)}),x_W^*\rangle=\langle p_t,x_W^*\rangle=g_t (x_W^*)+\frac{B_T}{T}$. More general assumptions such as convexity would not be enough for the proof to go through.\\
Considering that $f_t(x)$ is monotone DR-submodular for all $t\in[T]$, we can write:
\begin{align*}
	f_t (x_W^*)-f_t(x_t^{(k)}) &\overset{\text{(c)}}\leq f_t (x_W^*\vee x_t^{(k)})-f_t(x_t^{(k)})\\
	&\overset{\text{(d)}}\leq \langle \nabla f_t (x_t^{(k)}),(x_W^*\vee x_t^{(k)})-x_t^{(k)}\rangle\\
	&= \langle \nabla f_t(x_t^{(k)}),(x_W^*-x_t^{(k)})\vee 0\rangle\\
	&\overset{\text{(e)}}\leq \langle \nabla f_t(x_t^{(k)}),x_W^*\rangle
\end{align*}
where for $a,b\in \R^n$, $a\vee b$ denotes the entry-wise maximum of vectors $a$ and $b$, (c) and (e) are due to monotonocity of $f_t$ and (d) uses concavity of $f_t$ along non-negative directions.\\
Therefore, we conclude:
\begin{align*}
	\L_t (x_t^{(k+1)},\lambda_t)&\geq \L_t (x_t^{(k)},\lambda_t)+\frac{1}{K}\langle \nabla_x \L_t (x_t^{(k)},\lambda_t),v_t^{(k)}-x_W^*\rangle+\frac{1}{K}\big(f_t (x_W^*)-f_t(x_t^{(k)})\big)\\
	&-\frac{1}{K}\lambda_t g_t (x_W^*)-\frac{1}{K} \lambda_t \frac{B_T}{T}-\frac{LR^2}{2K^2}
\end{align*}
Equivalently, we can write:
\begin{align}
	\big(f_t(x_W^*)-f_t(x_t^{(k+1)})\big)&\leq (1-\frac{1}{K})\big(f_t(x_W^*)-f_t(x_t^{(k)})\big)-\lambda_t \big(g_t(x_t^{(k+1)})-g_t(x_t^{(k)})\big)+\frac{1}{K} \lambda_t g_t(x_W^*)\nonumber\\
	&+\frac{1}{K} \lambda_t \frac{B_T}{T}+\frac{LR^2}{2K^2}+\frac{1}{K}\langle \nabla \L_t(x_t^{(k)},\lambda_t),x_W^* -v_t^{(k)}\rangle\nonumber\\
	&=(1-\frac{1}{K})\big(f_t(x_W^*)-f_t(x_t^{(k)})\big)+\frac{1}{K} \big[-\lambda_t \langle p_t, v_t^{(k)}\rangle +\lambda_t g_t(x_W^*)+\lambda_t \frac{B_T}{T}\nonumber\\
	&+\frac{LR^2}{2K}+\langle \nabla \L_t(x_t^{(k)},\lambda_t),x_W^* -v_t^{(k)}\rangle\big]\label{reg1}
\end{align}
Replacing $t$ by $t+\tau$ in inequality $\ref{reg1}$ and taking the sum over $\tau\in\{0,\dots,W-1\}$ and $t\in\{1,\dots,T-W+1\}$, we obtain:
\begin{align}
	\sum_{t=1}^{T-W+1}\sum_{\tau=0}^{W-1}\big(f_{t+\tau}(x_W^*)-f_{t+\tau}(x_{t+\tau}^{(k+1)})\big)&\leq (1-\frac{1}{K})\sum_{t=1}^{T-W+1}\sum_{\tau=0}^{W-1}\big(f_{t+\tau}(x_W^*)-f_{t+\tau}(x_{t+\tau}^{(k)})\big)\nonumber\\
	&+\frac{1}{K} \sum_{t=1}^{T-W+1}\sum_{\tau=0}^{W-1}\big[-\lambda_{t+\tau} \langle p_{t+\tau}, v_{t+\tau}^{(k)}\rangle +\lambda_{t+\tau} g_{t+\tau}(x_W^*)\nonumber\\
	&+\lambda_{t+\tau} \frac{B_T}{T}+\frac{LR^2}{2K}+\langle \nabla \L_{t+\tau}(x_{t+\tau}^{(k)},\lambda_{t+\tau}),x_W^* -v_{t+\tau}^{(k)}\rangle\big]\label{reg2}
\end{align}
Applying inequality $\ref{reg2}$ recursively for all $k\in\{1,\dots,K\}$, we obtain:
\begin{align}
	\sum_{t=1}^{T-W+1}\sum_{\tau=0}^{W-1}\big(f_{t+\tau}(x_W^*)-f_{t+\tau}(\underbrace{x_{t+\tau}^{(K)}}_{=x_{t+\tau}})\big)&\leq \Pi_{k=0}^{K-1}(1-\frac{1}{K})\sum_{t=1}^{T-W+1}\sum_{\tau=0}^{W-1}\big(f_{t+\tau}(x_W^*)-f_{t+\tau}(x_{t+\tau}^{(0)})\big)\nonumber\\
	&+\sum_{k=0}^{K-1}\frac{1}{K} \Pi_{j=k+1}^{K-1}(1-\frac{1}{K}) \sum_{t=1}^{T-W+1}\sum_{\tau=0}^{W-1}\big[-\lambda_{t+\tau} \langle p_{t+\tau}, v_{t+\tau}^{(k)}\rangle \nonumber\\
	&+\lambda_{t+\tau} g_{t+\tau}(x_W^*)+\lambda_{t+\tau} \frac{B_T}{T}+\frac{LR^2}{2K}+\langle \nabla \L_{t+\tau}(x_{t+\tau}^{(k)},\lambda_{t+\tau}),x_W^* -v_{t+\tau}^{(k)}\rangle\big]\label{reg3}
\end{align}
Using the regret bound of Online Gradient Ascent instance $\Eps_k~\forall k\in[K]$, the following holds (Theorem $3.1.$ of \cite{hazan2016introduction}):
\begin{align*}
	\sum_{t=1}^T \langle \nabla_x \L_t (x_t^{(k)},\lambda_t),x_W^*-v_t^{(k)}\rangle &= \sum_{t=1}^T \langle \nabla_x \L_t (x_t^{(k)},\lambda_t),x_W^*\rangle -\sum_{t=1}^T \langle \nabla_x \L_t (x_t^{(k)},\lambda_t),v_t^{(k)}\rangle\\
	&\leq \max_x \sum_{t=1}^T \langle \nabla_x \L_t (x_t^{(k)},\lambda_t),x\rangle -\sum_{t=1}^T \langle \nabla_x \L_t (x_t^{(k)},\lambda_t),v_t^{(k)}\rangle\\
	&\leq \frac{R^2}{\mu}+\frac{\mu}{2}\sum_{t=1}^T \|\nabla_x \L_t (x_t^{(k)},\lambda_t)\|^2\\
	&= \frac{R^2}{\mu}+\frac{\mu}{2}\sum_{t=1}^T \|\nabla_x f_t (x_t^{(k)})-\lambda_t p_t\|^2\\
	&\overset{\text{(a)}}\leq \frac{R^2}{\mu}+\frac{\mu}{2}\sum_{t=1}^T \big(2\|\nabla_x f_t (x_t^{(k)})\|^2+2\lambda_t^2 \|p_t\|^2\big)\\
	&\overset{\text{(b)}}\leq \frac{R^2}{\mu}+\beta^2 \mu T+ \beta^2 \mu \sum_{t=1}^T\lambda_t^2
\end{align*}
where (a) uses the inequality $\|a+b\|^2\leq 2\|a\|^2+2\|b\|^2~\forall a,b\in \R^n$ and (b) is due to $\beta$-Lipschitzness of functions $f_t, g_t$ for all $t\in[T]$.\\
Using the inequality $(1-\frac{1}{K})^K \leq \frac{1}{e}$ in $\ref{reg3}$, we have:
\begin{align}
	\sum_{t=1}^{T-W+1}\sum_{\tau=0}^{W-1}\big(f_{t+\tau}(x_W^*)-f_{t+\tau}(x_{t+\tau})\big)&\leq \frac{1}{e}\sum_{t=1}^{T-W+1}\sum_{\tau=0}^{W-1}\big(f_{t+\tau}(x_W^*)-f_{t+\tau}(x_{t+\tau}^{(0)})\big)\nonumber\\
	&+\sum_{t=1}^{T-W+1}\sum_{\tau=0}^{W-1}\sum_{k=0}^{K-1}\frac{1}{K}\big[-\lambda_{t+\tau} \langle p_{t+\tau}, v_{t+\tau}^{(k)}\rangle +\lambda_{t+\tau} g_{t+\tau}(x_W^*)\nonumber\\
	&+\lambda_{t+\tau} \frac{B_T}{T}+\frac{LR^2}{2K}+\langle \nabla \L_{t+\tau}(x_{t+\tau}^{(k)},\lambda_{t+\tau}),x_W^* -v_{t+\tau}^{(k)}\rangle\big]\nonumber\\
	&=\frac{1}{e}\sum_{t=1}^{T-W+1}\sum_{\tau=0}^{W-1}\big(f_{t+\tau}(x_W^*)-\underbrace{f_{t+\tau}(0)}_{=0}\big)\nonumber\\
	&+\sum_{t=1}^{T-W+1}\sum_{\tau=0}^{W-1}\big[-\lambda_{t+\tau}g_{t+\tau}(x_{t+\tau})-\lambda_{t+\tau}\frac{B_T}{T} +\lambda_{t+\tau} g_{t+\tau}(x_W^*)\nonumber\\
	&+\lambda_{t+\tau} \frac{B_T}{T}+\frac{LR^2}{2K}+\sum_{k=0}^{K-1}\frac{1}{K}\langle \nabla \L_{t+\tau}(x_{t+\tau}^{(k)},\lambda_{t+\tau}),x_W^* -v_{t+\tau}^{(k)}\rangle\big]\label{reg4}
\end{align}
Rearranging the terms in $\ref{reg4}$, we obtain:
\begin{align}
	&\underbrace{\sum_{t=1}^{T-W+1}\sum_{\tau=0}^{W-1}\big((1-\frac{1}{e})f_{t+\tau}(x_W^*)-f_{t+\tau}(x_{t+\tau})\big)}_{\text{(a)}}+\underbrace{\sum_{t=1}^{T-W+1}\sum_{\tau=0}^{W-1}\lambda_{t+\tau}g_{t+\tau}(x_{t+\tau})}_{\text{(b)}}\leq\nonumber\\ &\underbrace{\sum_{t=1}^{T-W+1}\sum_{\tau=0}^{W-1}\lambda_{t+\tau} g_{t+\tau}(x_W^*)}_{\text{(c)}}+\sum_{k=0}^{K-1}\frac{1}{K}\underbrace{\sum_{t=1}^{T-W+1}\sum_{\tau=0}^{W-1}\langle \nabla L_{t+\tau}(x_{t+\tau}^{(k)},\lambda_{t+\tau}),x_W^* -v_{t+\tau}^{(k)}\rangle}_{\text{(d)}} +\frac{LR^2}{2K}W(T-W+1)\label{reg12}
\end{align}
(a) could be lower bounded as follows:
\begin{align}
	\text{(a)}&=WR_T-\sum_{i=1}^{W-1}(W-i)\big([(1-\frac{1}{e})f_{i}(x_W^*)-f_{i}(x_{i})]+[(1-\frac{1}{e})f_{T-i+1}(x_W^*)-f_{T-i+1}(x_{T-i+1})] \big)\nonumber\\
	&\geq WR_T-2F\sum_{i=1}^{W-1}(W-i)\nonumber\\
	&= WR_T-FW(W-1)\label{reg5}
\end{align}
Using Lemma 4.1 with $(1-\delta \mu^2)\leq 1$, we have:
\begin{align}
	\text{(b)}&= W\sum_{t=1}^{T}\lambda_t g_t(x_t)-\sum_{i=1}^{W-1}(W-i)\big(\lambda_i g_i (x_i)+\lambda_{T-i+1}g_{T-i+1}(x_{T-i+1})\big)\nonumber\\
	&\geq W\sum_{t=1}^{T}\lambda_t g_t(x_t)-\sum_{i=1}^{W-1}(W-i)\big(\mu(i-1)G^2+\mu(T-i)G^2\big)\nonumber\\
	&\geq W\sum_{t=1}^{T}\lambda_t g_t(x_t)-\frac{G^2}{2}\mu W(W-1)(T-1)\label{reg6}
\end{align} 
In order to bound (c), we use Lemma 4.2 and write:
\begin{align}
	\text{(c)}&\leq \sum_{t=1}^{T-W+1}\big(\lambda_t \underbrace{\sum_{\tau=0}^{W-1}g_{t+\tau}(x_W^*)}_{\leq 0}+\frac{G^2}{2}\mu W(W-1)\big)\nonumber\\
	&\leq \frac{1}{2}\mu G^2 W(W-1)(T-W+1)\label{reg9}
\end{align}
Finally, for a fixed $k\in[K]$, we can bound (d) as follows:
\begin{align}
	\text{(d)}&=W\sum_{t=1}^T \langle \nabla \L_{t}(x_{t}^{(k)}),x_W^* -v_{t}^{(k)}\rangle\nonumber\\
	&-\sum_{i=1}^{W-1}(W-i)\big([\underbrace{\langle \nabla \L_{i}(x_{i}^{(k)}),x_W^* -v_{i}^{(k)}\rangle}_{\geq -\beta R(1+\lambda_i)}]+[\underbrace{\langle \nabla \L_{T-i+1}(x_{T-i+1}^{(k)}),x_W^* -v_{T-i+1}^{(k)}\rangle}_{\geq -\beta R(1+\lambda_{T-i+1})}]\big)\nonumber\\
	&\leq\frac{R^2W}{\mu}+\beta^2 \mu TW+ \beta^2 \mu W \sum_{t=1}^T\lambda_t^2+\sum_{i=1}^{W-1}(W-i)\big(2\beta R+\beta R\underbrace{\lambda_i}_{\leq (i-1)\mu G}+\beta R\underbrace{\lambda_{T-i+1}}_{\leq (T-i)\mu G}\big)\nonumber\\
	&=\frac{R^2W}{\mu}+\beta^2 \mu TW+ \beta^2 \mu W \sum_{t=1}^T\lambda_t^2+\beta RW(W-1)+\frac{\beta RG}{2}\mu W(W-1)(T-1)\label{reg10}
\end{align}
Using the regret bound for Online Gradient Ascent (Theorem $3.1.$ of \cite{hazan2016introduction}), we have:
\begin{align}
	\sum_{t=1}^T \big(\L_t(x_t,\lambda_t)-\L_t(x_t,\lambda)\big)&=\sum_{t=1}^T\big(-\lambda_t g_t (x_t)+\frac{\delta \mu}{2}\lambda_t^2+\lambda g_t (x_t)-\frac{\delta \mu}{2}\lambda^2\big)\nonumber\\
	&\leq \frac{\lambda^2}{\mu}+\frac{\mu}{2}\sum_{t=1}^T \|\nabla_{\lambda}\L_t(x_t,\lambda_t)\|^2\nonumber\\
	&\leq \frac{\lambda^2}{\mu}+\frac{\mu}{2}\sum_{t=1}^T \big(-g_t (x_t)+\delta \mu \lambda_t\big)^2\nonumber\\
	&\overset{\text{(a)}}\leq \frac{\lambda^2}{\mu}+\frac{\mu}{2}\sum_{t=1}^T (2g_t^2 (x_t)+2\delta^2 \mu^2 \lambda_t^2)\nonumber\\
	&\leq \frac{\lambda^2}{\mu}+G^2 \mu T+\delta^2 \mu^3 \sum_{t=1}^T \lambda_t^2\label{reg11}
\end{align}
where we use $(a+b)^2 \leq 2a^2+2b^2~\forall a,b\in \R$ to derive inequality (a).\\
Combining $\ref{reg12}$, $\ref{reg5}$, $\ref{reg6}$, $\ref{reg9}$, $\ref{reg10}$ and $\ref{reg11}$, dividing both sides by $W$ and rearranging the terms, we conclude:
\begin{align*}
	R_T+C_T \lambda+\frac{\delta \mu}{2}\sum_{t=1}^T \lambda_t^2 -\frac{\delta \mu}{2}T\lambda^2-\frac{\lambda^2}{\mu}&\leq (F+\beta R)(W-1)+\frac{G}{2}(G+\beta R)\mu (W-1)(T-1)\\
	&+\frac{R^2}{\mu}+(G^2+\beta^2) \mu T+\frac{G^2}{2}\mu (W-1)(T-W+1)\\
	&+\frac{LR^2}{2K}(T-W+1)+(\delta^2 \mu^3+\beta^2 \mu)\sum_{t=1}^T \lambda_t^2\\
\end{align*}
Note that if $T$ is large enough such that $WT\geq 16R^2$ holds, we can write:
\begin{align*}
	\delta^2 \mu^2+\beta^2 &= 16\beta^4.\frac{R^2}{\beta^2 WT}+\beta^2\\
	&=\frac{16R^2}{WT}\beta^2+\beta^2\\
	&\leq 2\beta^2\\
	&=\frac{\delta}{2}
\end{align*}
Therefore, we can remove the terms $\sum_{t=1}^T \lambda_t^2$ from the inequality. Ignoring these terms, we obtain the desired result.

\end{document}